\documentclass[12pt]{amsart}
\usepackage[T1]{fontenc}
\usepackage[english]{babel}
\usepackage{amscd,amsmath,amsthm,amssymb,graphics}
\usepackage{lmodern,pst-node}
\usepackage{pstricks}
\usepackage{subcaption}
\usepackage{tikz}

\usetikzlibrary{positioning}
\usepackage{tikz,fullpage}
\usetikzlibrary{arrows,%
                petri,%
                topaths, automata}%
\usepackage{etex}

\usetikzlibrary{decorations.markings}
\tikzstyle{vertex}=[circle, draw, inner sep=0pt, minimum size=6pt]

\usepackage{multicol}
\usepackage{epic,eepic}
\usepackage{amsfonts,amssymb,amscd,amsmath,enumitem,verbatim}
\psset{unit=0.7cm,linewidth=0.8pt,arrowsize=2.5pt 4}
\newpsstyle{fatline}{linewidth=1.5pt}
\newpsstyle{fyp}{fillstyle=solid,fillcolor=verylight}
\definecolor{verylight}{gray}{0.97}
\definecolor{light}{gray}{0.9}
\definecolor{medium}{gray}{0.85}
\definecolor{dark}{gray}{0.6}

\def\frk{\frak}               % font for "Fraktur"

\def\Phi{{\frk n}}
\def\Phi{{\frk N}}
\def\opn#1#2{\def#1{\operatorname{#2}}} % to make operators
\opn\chara{char} \opn\length{\ell} \opn\pd{pd} \opn\rk{rk}
\opn\projdim{proj\,dim} \opn\injdim{inj\,dim} \opn\rank{rank}
\opn\depth{depth} \opn\grade{grade} \opn\height{height}
\opn\embdim{emb\,dim} \opn\codim{codim}

\opn\Tr{Tr} \opn\bigrank{big\,rank}
\opn\superheight{superheight}\opn\lcm{lcm}
\opn\trdeg{tr\,deg}%\emph{
\opn\reg{reg} \opn\lreg{lreg} \opn\ini{in} \opn\lpd{lpd}
\opn\size{size}\opn\bigsize{bigsize}
\opn\cosize{cosize}\opn\bigcosize{bigcosize}
\opn\sdepth{sdepth}\opn\sreg{sreg}
\opn\link{link}\opn\fdepth{fdepth}
\opn\deg{deg}
\opn\max{max}
\opn\indeg{indeg}
\opn\min{min}
\opn\psln{psln}
\opn\div{div} \opn\Div{Div} \opn\cl{cl} \opn\Cl{Cl}
\let\epsilon\varepsilon
\let\phi=\varphi
\let\kappa=\varkappa

\opn\Spec{Spec} \opn\Supp{Supp} \opn\supp{supp} \opn\Sing{Sing}
\opn\Ass{Ass} \opn\Min{Min}\opn\Mon{Mon} \opn\dstab{dstab} \opn\astab{astab}
\opn\Syz{Syz}
\opn\Ann{Ann} \opn\Rad{Rad} \opn\Soc{Soc}
\opn\Im{Im}
 \opn\Ind{Ind}
 \opn\del{del}
 \opn\Ker{Ker} \opn\Coker{Coker} \opn\Am{Am}
\opn\Hom{Hom} \opn\Tor{Tor} \opn\Ext{Ext} \opn\End{End}
\opn\Aut{Aut} \opn\id{id}

\opn\nat{nat}
\opn\pff{pf}%   \pf exists already
\opn\Pf{Pf} \opn\GL{GL} \opn\SL{SL} \opn\mod{mod} \opn\ord{ord}
\opn\Gin{Gin} \opn\Hilb{Hilb}\opn\sort{sort}
\opn\initial{init}
\opn\ende{end}
\opn\height{height}
\opn\bight{bight}
\opn\hte{ht}
\opn\indeg{indeg}
\opn\reg{reg}
\opn\depth{depth}
\opn\type{type}
\opn\ldim{ldim}
\opn\maxdeg{maxdeg}
\opn\aff{aff} \opn\con{conv} \opn\relint{relint} \opn\st{st}
\opn\lk{lk} \opn\cn{cn} \opn\core{core} \opn\vol{vol}
\opn\link{link} \opn\star{star}\opn\lex{lex}
\opn\gr{gr}

\def\pot#1#2{#1[\kern-0.28ex[#2]\kern-0.28ex]}

\opn\dirlim{\underrightarrow{\lim}}
\opn\inivlim{\underleftarrow{\lim}}
\def\Implies{\ifmmode\Longrightarrow \else
        \unskip${}\Longrightarrow{}$\ignorespaces\fi}
\def\implies{\ifmmode\Rightarrow \else
        \unskip${}\Rightarrow{}$\ignorespaces\fi}
\def\iff{\ifmmode\Longleftrightarrow \else
        \unskip${}\Longleftrightarrow{}$\ignorespaces\fi}

\let\:=\colon
 \theoremstyle{plain}
\newtheorem{Theorem}{Theorem}[section]
 \newtheorem{Lemma}[Theorem]{Lemma}
 \newtheorem{Corollary}[Theorem]{Corollary}
 \newtheorem{Proposition}[Theorem]{Proposition}

 \theoremstyle{definition}
 \newtheorem{Definition}[Theorem]{Definition}

 \newtheorem{Example}[Theorem]{Example}

\let\epsilon\varepsilon
\let\kappa=\varkappa
\def\qed{\ifhmode\textqed\fi
      \ifmmode\ifinner\quad\qedsymbol\else\dispqed\fi\fi}
\def\textqed{\unskip\nobreak\penalty50
       \hskip2em\hbox{}\nobreak\hfil\qedsymbol
       \parfillskip=0pt \finalhyphendemerits=0}
\def\dispqed{\rlap{\qquad\qedsymbol}}

\opn\dis{dis}
\def\pnt{{\raise0.5mm\hbox{\large\bf.}}}

\opn\Lex{Lex}
\begin{document}
 \title{A remark on sequentially Cohen-Macaulay monomial ideals}

 \author {Mozhgan Koolani and Amir Mafi*}

\address{Mozghan Koolani, Department of Mathematics, University of Kurdistan, P.O. Box: 416, Sanandaj,
Iran.}
\email{mozhgan.koolani@gmail.com}

\address{Amir Mafi, Department of Mathematics, University of Kurdistan, P.O. Box: 416, Sanandaj,
Iran.}
\email{a\_mafi@ipm.ir}


\keywords{Sequentially Cohen-Macaulay, Monomial ideals, Shedding vertex.\\
* Corresponding author}

\begin{abstract}
Let $R=K[x_1,\ldots,x_n]$ be the polynomial ring in $n$ variables over a field $K$.
We show that if $G$ is a connected graph with a basic $5$-cycle $C$, then $G$ is a sequentially Cohen-Macaulay graph if and only if there exists a shedding vertex $x$ of $C$ such that $G\setminus x$ and $G\setminus N[x]$ are sequentially Cohen-Macaulay graphs. Furthermore, we study the sequentially Cohen-Macaulay and Castelnuovo-Mumford regularity of square-free monomial ideals in some special cases.
\end{abstract}

\maketitle

\section*{Introduction}
Throughout this paper, we assume that $R=K[x_1,\ldots,x_n]$ is the polynomial ring in $n$ variables over a field $K$ and $I$ is a monomial ideal of $R$.
If $I$ is the square-free monomial ideal of $R$, we may consider the simplicial complex $\Delta$ over vertex set $V=\{x_1,\ldots,x_n\}$ for which $I=I_{\Delta}$ is the Stanley-Reisner ideal of $\Delta$ and $K[\Delta]=R/I_{\Delta}$ is the Stanley-Reisner ring. Note that the simplicial complex $\Delta$ on $V$ is a collection of subsets of $V$ such that: $(1)$ $\{x_i\}\in\Delta$ for $i=1,\ldots,n$, and $(2)$ if $A\in\Delta$ and $B\subseteq A$, then $B\in\Delta$. If $x$ is a vertex of the simplicial complex $\Delta$, then the {\it deletion} of $x$ from $\Delta$, denoted by $\del_{\Delta}(x)$, is the simplicial complex over the vertex set $V\setminus\{x\}$ with faces
$\{F: F\in \Delta, x\notin F\}$. The {\it link} of $x$ in $\Delta$, denoted by $\link_{\Delta}(x)$, is the subcomplex of $\del_{\Delta}(x)$ with faces
$\{F: F\in\del_{\Delta}(x), F\cup\{x\}\in\Delta\}$. It is clear that $I_{\del_{\Delta}(x)}=(I_{\Delta},x)$ and $I_{\link_{\Delta}(x)}=((I_{\Delta}:x),x)$.
We say a monomial ideal $I$ is Cohen-Macaulay (sequentially Cohen-Macaulay) when $R/I$ is Cohen-Macaulay (sequentially Cohen-Macaulay).
Stanley \cite{S} defined that a graded $R$-module $M$ is to be sequentially Cohen-Macaulay (i.e., SCM) if there exists a finite filtration of graded $R$-modules
$0=M_0\subset M_1\subset\cdots\subset M_r=M$ such that each $M_{i}/M_{i-1}$ is Cohen-Macaulay (i.e., CM) and the Krull dimension of the quotients are increasing: $\dim(M_1/M_0)<\dim(M_2/M_1)<\cdots<\dim(M_r/M_{r-1})$.
Note that every CM $R$-module is a SCM $R$-module. Moreover, it is known that $M$ is a CM $R$-module if and only if $M$ is an unmixed and a SCM $R$-module.

 Let $G$ be a simple graph (no loops or multiple edges) on the vertex set $V=\{x_1,\ldots,x_n\}$ and the edge set $E$. The {\it edge ideal} of the graph $G$ is the quadratic square-free monomial ideal $I(G)=(x_ix_j\mid \{x_i,x_j\}\in E)$ and it was first introduced by Villarreal \cite{Vi}. One can associated to $G$ the simplicial complex $\Delta_G$ called the {\it independence complex}, whose faces are the independent sets of the graph $G$. Note that the independent set in $G$ is the set with no two of its vertices are adjacent. The independence complex is the simplicial complex associated to $I(G)$ via the Stanley-Reisner correspondence. Hence we may consider the simplicial complex $\Delta_G$ for which $I(G)$ is the Stanley-Reisner ideal of $\Delta_G$. The graph $G$ is called SCM if $R/I(G)$ is SCM. The SCM of simplicial complexes and graphs are studied in \cite{AMS, D,FV,HH,HRW,Van,VV,W}.

In this paper we show that if $G$ is a connected graph with a basic $5$-cycle $C$, then $G$ is a SCM graph if and only if there exists a shedding vertex $x$ of $C$ such that $G\setminus x$ and $G\setminus N[x]$ are SCM graphs, where $N[x]=N(x)\cup\{x\}$ such that $N(x)$ is the neighborhoods set of $x$ and $G\setminus x$ is the induced subgraph of $G$ over the vertex set $V\setminus\{x\}$. Moreover, we study SCM and Castelnuovo-Mumford regularity of square-free monomial ideals in some special cases.

For any unexplained notion or terminology, we refer the reader to \cite{HH1} and \cite{V}.

\section{The results}
We start this section by recalling the following definition and theorem.
Suppose that $\Delta$ be a simplicial complex. The pure $i$-skeleton of $\Delta$ is defined as: $\Delta^{[i]}=\langle\{F\in\Delta:\dim(F)=i\}\rangle; -1\leq i\leq \dim(\Delta)$.

\begin{Theorem}\cite[Theorem 3.3]{D}\label{T0} A simplicial complex $\Delta$ is SCM if and only if the pure $i$-skeleton $\Delta^{[i]}$ is CM for all $-1\leq i\leq\dim(\Delta)$.
\end{Theorem}

\begin{Proposition}\label{P0}
Let $I$ be a monomial ideal and $u$ be a monomial element of $R$. If $I$ is SCM, then $(I:u)$ is SCM.
\end{Proposition}

\begin{proof}
By via polarization \cite[Proposition 4.11]{F}, we may assume that $I$ and $u$ are square-free. Now we assume that $\Delta$ is a simplicial complex of $I$. Since $I$ is SCM, by Theorem \ref{T0} we have $\Delta^{[i]}$ is CM for all $i$ and so by \cite[Proposition 6.3.15]{V} $\link_{\Delta^{[i]}}(u)$ is CM. By using \cite[Proposition 6.3.17]{V} we have $\link_{\Delta^{[i]}}(u)=(\link_{\Delta})^{[i-1]}(u)$ and it follows that $(\link_{\Delta})^{[i-1]}(u)$ is CM for all $i>1$. Hence by Theorem \ref{T0}  $\link_{\Delta}(u)$ is SCM. Since $\link_{\Delta}(u)=((I:u),u)$, it follows that $(I:u)$ is SCM.
\end{proof}

Recall that an ideal $I$ is called unmixed if all prime ideal of $\Ass(I)$ have the same height. A vertex $x$ of a simplicial complex $\Delta$ is called a shedding vertex when no face of $\link_{\Delta}(x)$ is a facet (maximal face) of $\del_{\Delta}(x)$.
\begin{Proposition}\label{P1}
Let $\Delta$ be a simplicial complex and $x$ be a shedding vertex. If $I_{\Delta}=I$ is unmixed, then $(I:x)$ and $(I,x)$ are unmixed. In particular, $I_{\link_{\Delta}(x)}$ and $I_{\del_{\Delta}(x)}$ are unmixed.
\end{Proposition}

\begin{proof}
From the exact sequence \[0\longrightarrow R/(I:x)\overset{x}\longrightarrow R/I\longrightarrow R/(I,x)\longrightarrow 0,\]
 we conclude that $\Ass(I:x)\subseteq\Ass(I)\subseteq\Ass(I:x)\cup\Ass(I,x)$. 
Since $x$ is a shedding vertex, by using \cite[Proposition 2.1]{JY} we have $\Ass(I,x)\subseteq\Ass(I)$ and hence $\Ass(I)=\Ass(I:x)\cup\Ass(I,x)$. Now, since $I$ is unmixed it therefore follows that $(I:x)$ and $(I,x)$ are unmixed.
Since $I_{\del_{\Delta}(x)}=(I_{\Delta},x)$ and $I_{\link_{\Delta}(x)}=((I_{\Delta}:x),x)$, the result is clear.
\end{proof}

In the next result we use the following definition; a $5$-cycle $C$ of $G$ is called basic if $C$ does not contain two adjacent vertices of degree three or more in $G$, see \cite{CC}.

\begin{Lemma}\cite[Lemma 38]{CC}\label{L0}
Let $G$ be a connected graph. Then every vertex of degree at least $3$ in a basic $5$-cycle is a shedding vertex.
\end{Lemma}

Suppose that $x$ is a vertex of graph $G$. Then it is clear that $\link_{\Delta}(x)=\Delta_{G\setminus N[x]}$ and $\del_{\Delta}(x)=\Delta_{G\setminus x}$.

The following theorem is a generalization of \cite[Theorem 40]{CC}.
\begin{Theorem}\label{T1}
Let $G$ be a connected graph with a basic $5$-cycle $C$. Then $G$ is a SCM graph if and only if there exists a shedding vertex $x\in V(C)$ such that $G\setminus x$ and $G\setminus N[x]$ are SCM graphs.
\end{Theorem}

\begin{proof}
$(\Longrightarrow):$ We may assume that $C=(x_1,x_2,x_3,x_4,x_5)$. If $G=C$, then by \cite[Proposition 4.1]{FV} it follows that $G$ is SCM. Since each vertex is a shedding vertex, $G\setminus x_i$ and $G\setminus N[x_i]$ are path and edge for each $1\leq i\leq 5$, by \cite[Corollary 7]{W} it follows that $G\setminus x_i$ and $G\setminus N[x_i]$ are SCM graphs.
Now suppose that $G\neq C$. We may assume that $\deg(x_1)\geq 3$. Since $C$ is a basic $5$-cycle, then $\deg(x_2)=2=\deg(x_5)$. Also, we may assume that $\deg(x_3)=2$ and $\deg(x_4)\geq 2$. By Lemma \ref{L0}, $x_1$ is a shedding vertex. By \cite[Theorem 3.3]{VV}, we conclude that $G\setminus N[x_1]$ is SCM. Now, we will prove that $G_1=G\setminus x_1$ is SCM. Since $G$ is SCM, then $G_2=G\setminus N[x_2]$ and $G_3=G\setminus N[x_3,x_5]$ are SCM. Suppose that $F_1,\ldots,F_r$ and $H_1,\ldots,H_t$ are facets of $\Delta_{G_2}$ and $\Delta_{G_3}$, respectively. Take $F\in\mathcal{F}(\Delta_{G_1})$. If $x_2\in F$, then $F\setminus x_2\in \mathcal{F}(\Delta_{G_2})$ and there exists $F_i$ such that $F=F_i\cup\{x_2\}$, where $1\leq i\leq r$. If $x_2\notin F$, then $x_3\in F$ and $x_4\notin F$. Therefore $x_5\in F$. Thus, $F\setminus\{x_3,x_5\}\in \mathcal{F}(\Delta_{G_3})$ and so there exists $H_j$ such that $F=H_j\cup\{x_3,x_5\}$, where $1\leq j\leq t$. This implies that $\mathcal{F}(\Delta_{G_1})=\{F_i\cup\{x_2\},H_j\cup\{x_3,x_5\}: 1\leq i\leq r,1\leq j\leq t\}$.
Set $\Delta_1=\langle F_i\cup\{x_2\}: 1\leq i\leq r\rangle$ and $\Delta_2=\langle H_j\cup\{x_3,x_5\}: 1\leq j\leq t\rangle$. Consider the exact sequence
$$0\longrightarrow K[\Delta_1\cup\Delta_2]\longrightarrow K[\Delta_1]\oplus K[\Delta_2]\longrightarrow K[\Delta_1\cap\Delta_2]\longrightarrow 0.$$
By using Auslander-Buchsbaum Theorem and \cite[Corollary 3.2]{H} $\depth(K[\Delta_1\cap\Delta_2])^{[i]}\geq\depth (K[\Delta_1]\oplus K[\Delta_2])^{[i]}-1$ and since $(K[\Delta_1]\oplus K[\Delta_2])^{[i]}$ is CM, by Depth Formula \cite[Lemma 2.3.9]{V} we conclude that $K[\Delta_1\cup\Delta_2]^{[i]}$ is CM. Therefore, by Theorem \ref{T0}, $K[\Delta_1\cup\Delta_2]$ is SCM. Since $I_{\Delta_1\cup\Delta_2}=I_{\Delta_{G_1}}$ and $I_{\Delta_{G_1}}=(I,x_1)$, it therefore follows that $(I,x_1)$ is SCM, as required.\\
$(\Longleftarrow):$ It follows by \cite[Theorem 2.2]{JY}.
\end{proof}

\begin{Corollary}
Let $G$ be a connected unmixed graph with a basic $5$-cycle $C$. Then $G$ is a CM graph if and only if there exists a shedding vertex $x\in V(C)$ such that $G\setminus x$ and $G\setminus N[x]$ are CM graphs.
\end{Corollary}

\begin{proof}
It is known that if  $G$ is a connected unmixed graph, then $G$ is CM if and only if $G$ is SCM. Therefore the result follows by Theorem \ref{T1} and Proposition \ref{P1}. 
\end{proof}

It is known that $n$-cycle graphs are SCM if and only if $n=3,5$. So we give the following examples which show that having a basic $5$-cycle $C$ in $G$ is essential. By using \cite[Lemma 2.3]{AMS} and Macaulay2 \cite{GS} the following examples easily follow.
\begin{Example}
Let $I=(x_1x_2,x_1x_4,x_2x_3,x_3x_4,x_1x_5)$ be an edge ideal of a bipartite graph $G$. Then $G$ is SCM, but $G\setminus x_5$ is not SCM.
\end{Example}

\begin{Example}
Let $I=(x_1x_2,x_1x_4,x_1x_5,x_2x_3,x_2x_7,x_3x_4,x_4x_5,x_4x_7,x_6x_7)$  be an edge ideal of a graph $G$. Then $G$ is SCM and $x_5$ is a shedding vertex, but $G\setminus x_5$ is not SCM.
\end{Example}

The following lemma easily follows by \cite[Theorem 6.4.23]{V} and Auslander-Buchsbaum formula.
\begin{Lemma}\label{L1}
Let $I$ be a monomial ideal of $R$. If $I$ is SCM, then $\depth(R/I)=\min\{\dim(R/\frak{p}): \frak{p}\in\Ass(I)\}$.
\end{Lemma}

The following result is a generalization of \cite[Lemma 4.1(i)]{CHH}.
\begin{Corollary}\label{C1}
Let $I,J$ be two monomial ideals of $R$ such that $I:J$ is SCM. Then $\depth(R/I)\leq\depth(R/I:J)$.
\end{Corollary}

\begin{proof}
It is clear that $\Ass(I:J)\subseteq\Ass(I)$. Therefore by Lemma \ref{L1} we have $$\depth(R/I:J)=\min\{\dim R/\frak{p}: \frak{p}\in\Ass(I:J)\}$$  $$\geq\min\{\dim R/\frak{p}: \frak{p}\in\Ass(I)\}$$
$$\geq\depth(R/I).$$
\end{proof}

The following example shows that the condition on $I:J$ in Corollary \ref{C1} is essential. For the computation of the following example we use Macaulay2 \cite{GS}.
\begin{Example}
Let $R=K[x_1,x_2,x_3,x_4]$, $I=(x_1x_3,x_2x_4)$ and $J=(x_2x_3, x_1x_4)$. Then $I:J=(x_1,x_3)(x_2,x_4)$ such that $\depth(R/I)=2\nleqslant 1=\depth(R/I:J)$.

\end{Example}

In the following result, we use the concept of Castelnuovo-Mumford regularity of a graded $R$-module $M$ which is defined as $\reg(M)=\max\{j-i| ~\beta_{i,j}(M)\neq 0\}$.
\begin{Proposition}\label{P2}
Let $I$ be a square-free monomial ideal and $x$ is a shedding vertex such that $(I,x)$ is SCM. Then we have the following:
\begin{itemize}
\item[(i)] $\depth(R/I)=\min\{\depth(R/(I:x)),\depth(R/(I,x))\}$;
\item[(ii)] $\pd(R/I)=\max\{\pd(R/(I:x)),\pd(R/(I,x))\}$;
\item[(iii)] $\reg(R/I)=\max\{\reg(R/(I:x)),\reg(R/(I,x))+1\}$.
\end{itemize}
\end{Proposition}

\begin{proof}
$(i)$ Since $x$ is a shedding vertex, we have $\Ass(I)=\Ass(I:x)\cup\Ass(I,x)$ and since $(I,x)$ is SCM, by Lemma \ref{L1} we have \[\depth(R/(I,x))=\min\{\dim(R/\frak{p}): \frak{p}\in\Ass(I,x)\}\] \[\geq\min\{\dim(R/\frak{p}): \frak{p}\in\Ass(I)\}\geq\depth(R/I).\] On the other hand by using \cite[Lemma 4.1]{CHH} we have $\depth(R/I)\leq\depth(R/(I:x))$. It therefore follows that
$\depth(R/I)\leq\min\{\depth(R/(I:x)),\depth(R/(I,x))\}$. Conversely by using the Depth Formula \cite[Lemma 2.3.9]{V}
 on the exact sequence $$0\longrightarrow R/(I:x)\overset{x}\longrightarrow R/I\longrightarrow R/(I,x)\longrightarrow 0,$$ we have $\depth(R/I)\geq\min\{\depth(R/(I:x)),\depth(R/(I,x))\}$. This completes the proof of case $(i)$.\\
$(ii)$ It follows by $(i)$ and the Auslander-Buchsbaum Theorem.\\
$(iii)$ It follows by \cite[Theorem 4.2]{HW}.
\end{proof}

\begin{Definition}(\cite[Definition 1]{BC})
A vertex $x$ of $G$ is called codominated if there exists a vertex $y\in V\setminus\{x\}$ such that $N[y]\subseteq N[x]$.
\end{Definition}

The following lemma immediately follows from \cite[Lemma 6]{W} and \cite[Theorem 5]{CC}.
\begin{Lemma}\label{L2}
Let $G$ be a $C_5$-free graph. Then a vertex $x$ of $G$ is a shedding vertex if and only if it is codominated. In particular, if $G$ is a bipartite graph then a vertex $x$ of $G$ is a shedding vertex if and only if it is codominated.
\end{Lemma}

Two edges $\{x,y\}$ and $\{z,u\}$ of $G$ is called $3$-{\it disjoint} if the induced subgraph of $G$ on $\{x,y,z,u\}$ is disconnected or equivalently in the complement of $G$ the induced graph on  $\{x,y,z,u\}$ is four-cycle. A subset $A$ of edges of $G$ is called a pairwise $3$-disjoint set of edges in $G$ if each pair of edges of $A$ is $3$-disjoint. The maximum cardinality of all pairwise $3$-disjoint sets of edges in $G$ is denoted by $a(G)$.

\begin{Lemma}(\cite[Lemma 23]{BC})\label{L3}
If $x$ is a codominated vertex of a graph $G$, then $a(G\setminus x)\leq a(G)$ and $a(G\setminus N[x])+1\leq a(G)$.
\end{Lemma}

\begin{Lemma}(\cite[Lemma 2.10]{DHS})\label{L4}
Let $x$ be a vertex of a graph $G$. Then $\reg(R/I(G))\leq\max\{\reg(R/I(G\setminus x)),\reg(R/I(G\setminus N[x]))+1\}$.
Moreover, $\reg(R/I(G))$ always equals to one of $\reg(R/I(G\setminus x))$ or $\reg(R/I(G\setminus N[x]))+1$.
\end{Lemma}

\begin{Theorem}\label{T2}
Let $\mathcal{F}$ be a family of graphs such that every graph has a codominated vertex.
If $G\setminus x$ and $G\setminus N[x]$ are in $\mathcal{F}$ for $G\in \mathcal{F}$ and a codominated vertex $x$ of $G$,
 then $\reg(R/I(G))=a(G)$ for all $G\in \mathcal{F}$.
\end{Theorem}

\begin{proof} Let $G$ be an arbitrary element of $\mathcal{F}$. By \cite[Lemma 2.2]{K}, we have $\reg(R/I(G))\geq a(G)$. Now by induction on $\vert{V}\vert$, we prove that $\reg(R/I(G))\leq a(G)$.
If $\vert{V}\vert=2$, then $\reg(R/I(G))=1$ and also $a(G)=1$. Thus the result follows in this case. Suppose that $\vert{V}\vert\geq 2$. There exists a codominated vertex $x\in V$ such that $G\setminus x$ and $G\setminus N[x]$ are in $\mathcal{F}$. By Lemma \ref{L4} we have $\reg(R/I(G))\leq\max\{\reg(R/I(G\setminus x)),\reg(R/I(G\setminus N[x]))+1\}$ and by using induction hypothesis it follows $\reg(R/I(G))\leq\max\{a(G\setminus x),a(G\setminus N[x])+1\}$.
Now by Lemma \ref{L3} we have $\reg(R/I(G))\leq a(G)$. This completes the proof, as required.
\end{proof}

By using Theorem \ref{T2}, we readily conclude the following known result.
\begin{Corollary}\label{C2}
Let $G$ be a graph such that one of the following conditions satisfies:
\begin{itemize}
\item[(i)] \cite[Corollary 6.9]{HV} $G$ is a chordal graph;
\item[(ii)] \cite[Theorem 3.2]{Van} $G$ is a SCM bipartite graph;
\item[(iii)] \cite[Theorem 2.4]{KM} $G$ is a $C_5$-free vertex decomposable;
\item[(iv)] \cite[Theorem 11]{T} $G$ is a Cameron-Walker graph.
\item[(v)] \cite[Lemma 3.4]{MMCRTY} $G$ is a very well-covered Cohen-Macaulay graph.
\end{itemize}
Then $\reg(R/I(G))=a(G)$.
\end{Corollary}

\begin{proof}
If $G$ satisfies in $(i),(ii),(iii)$ and  $(iv)$, then it has a codominated vertex and also in case $(v)$ $G$ has a codominated vertex by proof of \cite[Theorem 3.2]{MMCRTY}. Thus by Theorem \ref{T2} the result follows.
\end{proof}

\subsection*{Acknowledgements}
We would like to thank H. Hassanzadeh for some useful examples. We also thank the referee for their meticulous comments that helped us to improve this manuscript.

%%%%%%%%%%%%%%%%%%%%%%%%%%%%%%%%%%%%%%%%%%%%%%%%%%%%%%%%%%%%%%%%%%%%%%%%%%%%%

\end{document}